\documentclass[11pt,reqno]{amsart}
\usepackage{hyperref}
\usepackage{url}
\usepackage{amsmath,amssymb}
\usepackage{tikz,enumerate}
\usepackage{diagbox}
\usepackage{appendix}
\usepackage{epic}

\newtheorem{theorem}{Theorem}
\newtheorem{corollary}[theorem]{Corollary}
\newtheorem{conj}[theorem]{Conjecture}
\newtheorem{lemma}[theorem]{Lemma}

\theoremstyle{definition}
\newtheorem{defn}[theorem]{Definition}
\newtheorem{rem}{Remark}
\numberwithin{equation}{section}
\numberwithin{theorem}{section}
 \numberwithin{rem}{section}
\usepackage{tikz,enumerate}
\allowdisplaybreaks[4]
\begin{document}

\title[Arithmetic Properties of Odd Ranks and $k$-Marked Odd Durfee Symbols]
 {Arithmetic Properties of Odd Ranks and $k$-Marked Odd Durfee Symbols}

\author{Liuquan Wang}

\address{School of Mathematics and Statistics, Wuhan University, Wuhan 430072, Hubei, P.R. China}
%\address{Hubei Key Laboratory of Computational Science (Wuhan University), Wuhan 430072, Hubei, P.R. China.}
\email{wanglq@whu.edu.cn;mathlqwang@163.com}



\dedicatory{}

\begin{abstract}
Let $N^{0}(m,n)$ be the number of odd Durfee symbols of $n$ with odd rank $m$, and $N^{0}(a,M;n)$ be the number of odd Durfee symbols of $n$ with odd rank congruent to $a$ modulo $M$.  We  give explicit formulas for the generating functions of $N^{0}(a,M;n)$ and their $\ell$-dissections where $0\le a \le M-1$ and $M, \ell \in  \{2, 4, 8\}$. From these formulas, we obtain some interesting arithmetic properties of $N^{0}(a,M;n)$. Furthermore, let $\mathcal{D}_{k}^{0}(n)$ denote the number of $k$-marked odd Durfee symbols of $n$.  Andrews (2007) conjectured that $\mathcal{D}_{2}^{0}(n)$ is even if $n\equiv 4$ or 6 (mod 8) and $\mathcal{D}_{3}^{0}(n)$ is even if $n\equiv 1, 9, 11$ or 13 (mod 16). Using our results on odd ranks, we prove Andrews' conjectures.
\end{abstract}

\maketitle
\section{Introduction and Main Results}

Given a positive integer $n$, a partition $\alpha=(a_1, a_2, \dots, a_r)$  of $n$ is a non-increasing sequence of positive integers that add up to $n$. The $a_i$'s are called the parts of $\alpha$. We denote by $l(\alpha):=r$ the number of parts in $\alpha$ and call $|\alpha|:=a_1+a_2+\cdots +a_r$ the weight of $\alpha$.  As usual, let $p(n)$ denote the number of partitions of $n$ and we define $p(0)=1$ for convention. Its generating function satisfies
\begin{align}\label{p-gen}
\sum_{n=0}^{\infty}p(n)q^n=\frac{1}{(q;q)_{\infty}},
\end{align}
where
\begin{align}
(a;q)_{n}:=\prod\limits_{k=0}^{n-1}(1-aq^{k})
\end{align}
and
\begin{align}
(a;q)_{\infty}:=\lim\limits_{n\rightarrow \infty}(a;q)_{n}, \quad |q|<1.
\end{align}

Ramanujan \cite{Ramanujan-1919,Ramanujan-1920} discovered the following beautiful congruences:
\begin{align}
p(5n+4)&\equiv 0 \pmod{5}, \label{pn-mod5}\\
p(7n+5)&\equiv 0 \pmod{7}, \label{pn-mod7} \\
p(11n+6)&\equiv 0 \pmod{11}. \label{pn-mod11}
\end{align}
These congruences are now  known as the Ramanujan congruences. Since their appearance, there have been numerous studies concerning arithmetic properties of combinatorial functions. These functions usually appear as coefficients for some $q$-series arising from combinatorics. There is no general method for proving congruences for such $q$-series, except in some cases such as  when the $q$-series is closely related to modular forms. In most cases, we need clever $q$-series manipulations as well as $q$-series identities. For more introduction on this background, see  \cite{Andrews-book} or \cite{Berndt}.

In order to explain Ramanujan's congruences \eqref{pn-mod5}-\eqref{pn-mod11} combinatorially, Dyson \cite{Dyson} introduced a statistic called the  rank. The rank of a partition is defined as its largest part minus the number of parts. Let $N(m,n)$ denote the number of partitions of $n$ with rank $m$, and  let $N(a,M;n)$ denote the number of partitions of $n$ with rank $\equiv a$ (mod $M$). It is well known that
\begin{align}
R_{1}(z;q)&:= 1+\sum_{m=-\infty}^{\infty}\sum_{n=1}^{\infty}N(m,n)z^mq^n \nonumber \\
&=\sum_{n=0}^{\infty}\frac{q^{n^2}}{(zq;q)_{n}(z^{-1}q;q)_{n}}.\label{rank-gen}
\end{align}
 Dyson \cite{Dyson} conjectured that for $(r,M)\in \{(4,5), (5,7) \}$, $0\le a \le M-1$ and $n\ge 0$,
\begin{align}
N(a,M;Mn+r)=\frac{1}{M}p(Mn+r). \label{N-mod57}
\end{align}
Atkin and Swinnerton-Dyer \cite{AS-D} proved \eqref{N-mod57} by studying the generating functions of $N(a,M;n)$ with $M=5$ or 7. The arithmetic relations in \eqref{N-mod57} gave a satisfactory explanation to the congruences \eqref{pn-mod5} and \eqref{pn-mod7}. To give a similar combinatorial interpretation to the congruence \eqref{pn-mod11}, another combinatorial quantity, the crank, was hypothesized by Dyson \cite{Dyson} and finally
found explicitly by Andrews and Garvan \cite{Andrews-Garvan}.

Around 2003, Atkin and Garvan \cite{Atkin-Garvan} considered the $k$-th moment of the rank which is defined as
\begin{align}\label{k-th-moment}
N_{k}(n):=\sum_{m=-\infty}^{\infty}m^{k}N(m,n).
\end{align}
Since $N(-m,n)=N(m,n)$, all the odd order moments are zero.  Andrews \cite{Andrews} discovered that there is a rich combinatorial and enumerative structure associated with the moments of ranks. He considered a symmetrized $k$-th moment function (see \cite[Eq.\ (1.13)]{Andrews})
\begin{align}\label{eta-k-rank}
\eta_{k}(n):=\sum_{m=-\infty}^{\infty}\binom{m+\lfloor \frac{k-1}{2}\rfloor}{k}N(m,n).
\end{align}
Again, it is not difficult to see that $\eta_{2k+1}(n)=0$ (see \cite[Theorem 1]{Andrews}). By introducing the concepts of Durfee symbols and $k$-marked Durfee symbols, Andrews gave a combinatorial interpretation for $\eta_{2k}(n)$ by showing that $\eta_{2k}(n)$ equals the number of $(k+1)$-marked Durfee symbols of $n$.  For our purpose, we will not discuss Durfee symbols and $k$-marked Durfee symbols.

In \cite{Andrews} Andrews also introduced odd Durfee symbols and $k$-marked odd Durfee symbols. The motivation for Andrews to introduce odd Durfee symbols is to give a natural combinatorial explanation to an identity associated with Watson's third order mock theta function $\omega(q)$ \cite{Watson}, which is defined as
\begin{align}
\omega(q):=\sum_{n=0}^{\infty}\frac{q^{2n(n+1)}}{(q;q^2)_{n+1}^2}. \label{omega-defn}
\end{align}
Fine \cite[Eq.\ (26.84)]{Fine} discovered that $\omega(q)$ satisfies
\begin{align}
q\omega(q)=\sum_{n=0}^{\infty}\frac{q^{n+1}}{(q;q^2)_{n+1}}. \label{Fine-formula}
\end{align}
Here, the right-hand side of \eqref{Fine-formula} is the generating function for $p_{\omega}(n)$, the number of partitions of $n$ wherein at least all but one instance of the largest part is one of a pair of consecutive non-negative integers.
%The right side of the above identity is the generating function of certain partitions. To be more specific, let $p_{\omega}(n)$ be the number of partitions of $n$ wherein at least all but one instance of the largest part is one of a pair of consecutive non-negative integers. Then \eqref{Fine-formula} can be restated as
%\begin{align}
%q\omega(q)=\sum_{n=0}^{\infty}p_{\omega}(n)q^n. \label{Fine-formula-2nd}
%\end{align}
By utilizing MacMahon's modular partitions with modulus 2, Andrews showed that each partition enumerated by $p_{\omega}(n)$ has associated with it an odd Durfee symbol of $n$.
\begin{defn}\label{odd-durfee-defn}
An \textit{odd Durfee symbol of $n$} is a two-rowed array with a subscript of the form
\begin{align*}
\begin{pmatrix}
a_1 & a_2 & \cdots & a_s \\
b_1 &b_2 &\cdots & b_t
\end{pmatrix}_{D}
\end{align*}
wherein all entries are odd numbers such that
\begin{enumerate}
\item $2D+1\ge a_1 \ge a_2 \ge \cdots \ge a_s >0$; \label{odd-defn-1}
\item $2D+1 \ge b_1 \ge b_2 \ge \cdots \ge b_t>0$; and  \label{odd-defn-2}
\item $n=\sum_{i=1}^{s}a_{i}+\sum_{j=1}^{t}b_{j}+2D^2+2D+1$. \label{odd-defn-3}
\end{enumerate}
The odd rank of an odd Durfee symbol is the number of entries in the top row minus the number of entries in the bottom row.
\end{defn}
Definition \ref{odd-durfee-defn} is an equivalent version of Andrews's  definition \cite[pp.\ 62-63]{Andrews} and is due to Ji \cite[Definition 1.2]{Ji}.

Let $N^{0}(m,n)$ denote the number of odd Durfee symbols of $n$ with odd rank $m$. By interchanging the rows of the symbol, it is clear that
\begin{align}\label{N-odd-even}
N^{0}(m,n)=N^{0}(-m,n).
\end{align}
Andrews \cite[Sec.\ 8]{Andrews} proved that
\begin{align}
p_{\omega}(n)=\sum_{m=-\infty}^{\infty}N^{0}(m,n).
\end{align}
and  \cite[Eq.\ (8.3)]{Andrews}
\begin{align}
R_{1}^{0}(z;q):=&\sum_{n=1}^{\infty}\sum_{m=-\infty}^{\infty}N^{0}(m,n)z^mq^n \nonumber\\
=&\sum_{n=0}^{\infty}\frac{q^{2n(n+1)+1}}{(zq;q^2)_{n+1}(z^{-1}q;q^2)_{n+1}}. \label{odd-rank-gen}
\end{align}
Clearly, if we let $z=1$ in \eqref{odd-rank-gen}, we recover Fine's formula \eqref{Fine-formula}.

Recently, Andrews, Dixit and Yee \ \cite{ADY} showed that $p_{\omega}(n)$ also counts the number of partitions of $n$ in which all odd parts are less than twice the smallest part. Besides its rich combinatorial meanings, the function $p_{\omega}(n)$ also satisfies many interesting congruences. Using the theory of modular forms, Waldherr \cite[Theorem 1.1]{Waldherr} gave the first explicit example of congruences satisfied by $p_{\omega}(n)$. Together with several other congruences, he proved that for any integer $n\ge 0$,
\begin{align}
p_{\omega}(40n+r)\equiv & 0 \pmod{5}, \quad r\in \{28,36\}. \label{p-omega-mod5}
\end{align}
An elementary proof of \eqref{p-omega-mod5} was given by Andrews, Passary, Sellers and Yee \cite{APSY}. They also established the following congruences:
\begin{align}
p_{\omega}(8n+4) &\equiv 0 \pmod{4}, \label{p-omega-1}\\
p_{\omega}(8n+6) &\equiv 0 \pmod{8}, \quad \text{\rm{and}}\label{p-omega-2} \\
 p_{\omega}(16n+13) &\equiv 0 \pmod{4}. \label{p-omega-3}
\end{align}
For more congruences satisfied by $p_{\omega}(n)$, we refer the reader to \cite{ADY,APSY,Waldherr,Wang}.
Clearly, the roles of the odd rank $N^{0}(m,n)$ and $p_{\omega}(n)$ are similar to the roles of the rank $N(m,n)$ and $p(n)$, and knowing the properties of odd rank will be helpful to understand $p_{\omega}(n)$.

%By analyzing the generating functions \eqref{rank-gen} and \eqref{odd-rank-gen}, we find that the odd rank is closely related to the ordinary rank.
%\begin{theorem}\label{rank-odd}
%Let $n\ge 1$ and $m$ be integers.
%\begin{enumerate}[$(1)$]
%\item If $n \equiv m$ \text{\rm{(mod 2)}}, then
%\begin{align}
%N^{0}(m,n)=0.
%\end{align}
%\item If $n\equiv m+1$ \text{\rm{(mod 2)}}, then we have
%\begin{align}\label{recurrence}
%N^{0}(m,n)-N^{0}(m-1,n-1)=N\left(m+1,\frac{n-m-1}{2}\right)
%\end{align}
%and
%\begin{align}\label{odd-exp-rank}
%N^{0}(m,n)=\sum_{k=0}^{n-1}N\left(m+1-k,\frac{n-m-1}{2} \right).
%\end{align}
%\end{enumerate}
%\end{theorem}
%An interesting consequence is
%\begin{corollary}\label{cor-rank-pn}
%If $n\le 3m+5$ and $n\equiv m+1$ \text{\rm{(mod 2)}}, we have
%\begin{align}\label{odd-pn}
%N^{0}(m,n)=p\left(\frac{n-m-1}{2} \right).
%\end{align}
%\end{corollary}
%From Theorem \ref{rank-odd}, we observe that studying the properties of odd ranks is also helpful in exploring the properties of the ordinary ranks.

Similar to \eqref{eta-k-rank}, Andrews \cite[p.\ 63]{Andrews} considered a symmetrized $k$-th moment function
\begin{align}\label{eta-defn}
\eta_{k}^{0}(n):=\sum_{m=-\infty}^{\infty}\binom{m+\lfloor \frac{k}{2}\rfloor}{k}N^{0}(m,n).
\end{align}
Since $N^{0}(m,n)=N^{0}(-m,n)$, we have $\eta_{2k+1}^{0}(n)=0$ \cite[p.\ 63]{Andrews}.
Using Watson's first identity on page 66 of \cite{Watson}, we have \cite[Eq.\ (8.5)]{Andrews}
\begin{align}\label{R-defn}
R_{1}^{0}(z;q)=\frac{1}{(q^2;q^2)_{\infty}}\sum_{n=-\infty}^{\infty}\frac{(-1)^nq^{3n^2+3n+1}}{1-zq^{2n+1}}.
\end{align}
Using \eqref{R-defn}, Andrews deduced that \cite[Theorem 21]{Andrews}
\begin{align}
\sum_{n=1}^{\infty}\eta_{2k}^{0}(n)q^n=\frac{1}{(q^2;q^2)_{\infty}}\sum_{n=-\infty}^{\infty}\frac{(-1)^nq^{3n^2+(2k+3)n+k+1}}{(1-q^{2n+1})^{2k+1}}.
\end{align}

To give a combinatorial explanation of \eqref{eta-defn}, Andrews introduced $k$-marked odd Durfee symbols. See \cite[p.\ 62]{Andrews} or \cite[Definition 4.3]{Ji} for definition.
%\begin{defn}\label{k-odd-durfee-defn}
%(Cf.\ \cite[Definition 4.3]{Ji}.) A $k$-marked odd Durfee symbol of $n$ is composed of $k$ pairs of partitions into odd parts with a subscript, given by
%\begin{align*}
%\eta^{0}=\begin{pmatrix} \alpha^k, \alpha^{k-1}, \dots, \alpha^1 \\ \beta^{k}, \beta^{k-1}, \dots, \beta^{1} \end{pmatrix}_{D},
%\end{align*}
%where $\alpha^{i}$ (resp. $\beta^{i}$) are all partitions with odd parts such that
%\begin{align}
%\sum_{i=1}^{k}\left(\left|\alpha^{i}\right|+\left|\beta^{i}\right| \right)+2D^2+2D+1=n,
%\end{align}
%and the following conditions are satisfied:
%\begin{enumerate}
%\item For $1\le i<k$, $\alpha^{i}$ is a nonempty partition, while $\alpha^{k}$ and $\beta^{i}$ could be empty;
%\item $\beta_{1}^{i-1}\le \alpha_{1}^{i-1} \le \beta_{l(\beta^{i})}^{i}$ for $2\le i\le k$;
%\item $\beta_{1}^k,\alpha_{1}^k\le 2D+1$.
%\end{enumerate}
%Here $\alpha_{1}^{i}$ (resp.\ $\beta_{1}^{i}$) denotes the largest part of the partition $\alpha^{i}$ (resp.\ $\beta^{i}$), and $\alpha_{l(\alpha^i)}^{i}$ (resp.\ $\beta_{l(\beta^i)}^{i}$) denotes the smallest part of the partition $\alpha^{i}$ (resp.\ $\beta^{i}$).
%\end{defn}
Let $\mathcal{D}_{k}^{0}(n)$ denote the number of $k$-marked odd Durfee symbols of $n$.  Andrews \cite[Corollary 29]{Andrews} proved that for $k\ge 0$,
\begin{align}\label{D-eta}
\mathcal{D}_{k+1}^{0}(n)=\eta_{2k}^{0}(n).
\end{align}
Andrews also investigated the parity of $\mathcal{D}_{k}^{0}(n)$. He proved that for each $k\ge 1$, if $n\equiv k-1$ (mod 2), then $\mathcal{D}_{k}^{0}(n)$ is even. Andrews \cite[Conjectures A and B]{Andrews} then proposed the following conjectures.
\begin{conj}\label{conj-1}
$\mathcal{D}_{2}^{0}(n)$ is even if $n\equiv 4$ or $6$ \text{\rm{(mod 8)}}.
\end{conj}
\begin{conj}\label{conj-2}
$\mathcal{D}_{3}^{0}(n)$ is even if $n\equiv 1, 9, 11$ or $13$ \text{\rm{(mod 16)}}.
\end{conj}

One of the main goals of this paper is to confirm the above conjectures. To prove these results, via \eqref{D-eta} it suffices to prove the corresponding congruences for $\eta_{2}^{0}(n)$ and $\eta_{4}^{0}(n)$. This observation will eventually turn our attention to the arithmetic properties of $N^{0}(a,M;n)$ with $0\le a \le M-1$ and $M\in \{4,8\}$ after analyzing \eqref{eta-defn}, where $N^{0}(a,M;n)$ denotes the number of odd Durfee symbols of $n$ with odd rank congruent to $a$ modulo $M$. Hence, we need to study the odd ranks modulo 4 and 8.
Santa-Gadea and Lewis \cite{Lewis} proved a number of results on ranks and cranks modulo 4 and 8. Recently, Andrews, Berndt, Chan, Kim and Malik \cite{ABCKM} found some results on the ranks modulo 4 and 8. For example, they showed that \cite[(7.5), (7.6)]{ABCKM}
\begin{align}
N(0,4;2n)-N(2,4;2n)=&(-1)^n\left(N(0,8;2n)-N(4,8;2n) \right), \label{Chan-1} \\
N(0,4;2n+1)-N(2,4;2n+1)=&(-1)^n\Big( N(0,8;2n+1)+N(1,8;2n+1) \nonumber \\
&-2N(3,8;2n+1)-N(4,8;2n+1)\Big). \label{Chan-2}
\end{align}
Later Mortenson \cite{Mortenson} used different methods to prove these results and obtained generating functions for $N(a,M;n)-p(n)/M$ where $M\in \{4,8\}$. Motivated by their works, we are able to give explicit formulas for generating functions associated with $N^{0}(a,M;\ell n+r)$, where $0\le a \le M-1$, $0\le r \le \ell-1$ and $M, \ell \in \{2,4,8\}$  (see Theorems \ref{thm-mod2}--\ref{thm-mod8}). Using these generating functions, we prove some interesting arithmetic relations analogous to \eqref{N-mod57}, \eqref{Chan-1} and \eqref{Chan-2}.
\begin{theorem}\label{N-relation}
For any integer $n\ge 0$ we have
\begin{align}
N^{0}(0,8;8n+r)&=N^{0}(4,8;8n+r), \quad r\in \{5,7\}, \label{N08-48-8n57} \\
N^{0}(1,8;8n+r)&=N^{0}(3,8;8n+r), \quad r \in \{4,6\}. \label{N18-38-8n46}
\end{align}
\end{theorem}
Theorem \ref{N-relation} leads to the following corollary.
\begin{corollary}\label{cor-N-relation}
For any integer $n\ge 0$ we have
\begin{align}
N^{0}(0,4;8n+r)&=2N^{0}(0,8;8n+r), \quad r\in \{5,7\}, \label{N04-08-8n57} \\
N^{0}(1,4;8n+r)&=2N^{0}(1,8;8n+r), \quad r\in \{4,6\}. \label{N14-18-8n46}
\end{align}
\end{corollary}

Meanwhile, we give simple formulas for the generating functions of certain odd rank differences. To state these formulas, let
$J_{m}:=(q^m;q^m)_{\infty}$ and we recall a universal mock theta function
\begin{align}\label{g-defn}
g(x;q):=x^{-1}\left(-1+\sum_{n=0}^{\infty}\frac{q^{n^2}}{(x;q)_{n+1}(q/x;q)_{n}} \right).
\end{align}
\begin{theorem}\label{rank-differences}
We have
\begin{align}
\sum_{n=0}^{\infty}\left(N^{0}(0,8;8n+1)-N^{0}(4,8;8n+1) \right)q^n=&\frac{J_{2}^4}{J_{1}^2J_{4}}, \label{N08-48-8n1-diff} \\
\sum_{n=0}^{\infty}\left(N^{0}(0,8;8n+3)-N^{0}(4,8;8n+3) \right)q^n=&\frac{J_{4}^3}{J_{2}^2}+qg(q^2;q^4), \label{N08-48-8n3-diff} \\
\sum_{n=0}^{\infty}\left(N^{0}(1,8;8n)-N^{0}(3,8;8n) \right)q^n=&qg(q;q^4), \label{N18-38-8n-diff} \\
\sum_{n=0}^{\infty}\left(N^{0}(1,8;8n+2)-N^{0}(3,8;8n+2) \right)q^n=&\frac{J_{2}J_{4}}{J_{1}}. \label{N18-38-8n2-diff}
\end{align}
\end{theorem}
A consequence of this theorem is the following remarkable identity.
\begin{corollary}\label{N-p-omega}
For any integer $n\ge 1$ we have
\begin{align}\label{N-pomega-16n5}
N^{0}(0,8;16n-5)-N^{0}(4,8;16n-5)=p_{\omega}(n).
\end{align}
\end{corollary}
As a supplementary result to Theorem \ref{N-relation}, by Theorem \ref{rank-differences} we prove some strict inequalities between odd ranks.
\begin{theorem}\label{rank-inequality}
 For any integer $n\ge 0$ ($n\ge 1$ when $r=0$) we have
\begin{align}
N^{0}(0,8;8n+r)&>N^{0}(4,8;8n+r), \quad r\in \{1,3\}, \label{N08-48-8n13-ineq} \\
N^{0}(1,8;8n+r)&>N^{0}(3,8;8n+r), \quad r\in \{0,2\}. \label{N18-38-8n02-ineq}
\end{align}
\end{theorem}
Finally, we present some congruences satisfied by the odd ranks.
\begin{theorem}\label{cong-thm}
Let $k$ be a positive integer. For any integer $n\ge 0$ we have
\begin{align}
N^{0}(k,2k;n)&\equiv 0 \pmod{2}, \label{Nk-2k-cong} \\
N^{0}(0,2;16n+9)&\equiv 0 \pmod{2}, \label{N02-16n9-cong}\\
%N^{0}(1,2;4n+2)& \equiv 0 \pmod{2}, \label{N12-4n2-cong}\\
N^{0}(1,2;8n+4)&\equiv 0 \pmod{4}, \label{N12-8n4-cong} \\
N^{0}(1,2;8n+6)& \equiv 0 \pmod{8}, \label{N12-8n6-cong}\\
N^{0}(1,2;40n+r)&\equiv 0 \pmod{5}, \quad r\in \{28, 36\}, \label{N12-mod5-cong} \\
N^{0}(0,4;8n+r) &\equiv 0 \pmod{4}, \quad r\in \{5,7\},  \label{N04-8n57-cong}  \\
N^{0}(0,4;16n+13) &\equiv 0 \pmod{8}, \label{N04-16n13-cong}\\
N^{0}(0,4;16n+15) &\equiv 0 \pmod{32}, \label{N04-16n15-cong} \\
N^{0}(0,4;32n+23)&\equiv 0 \pmod{16}, \label{N04-32n23-cong} \\
N^{0}(0,4;32n+31) &\equiv 0 \pmod{512}, \label{N04-32n31-cong} \\
%N^{0}(1,4;8n+4) &\equiv 0 \pmod{2},  \label{N14-8n4-cong} \\
%N^{0}(1,4;8n+6)&\equiv 0 \pmod{4}, \label{N14-8n6-cong} \\
%N^{0}(2,4;4n+1) &\equiv 0 \pmod{2}, \label{N24-4n1-cong}\\
%N^{0}(2,4;8n+1) &\equiv 0 \pmod{8}, \label{N24-8n1-cong} \\
%N^{0}(2,4;8n+r) &\equiv 0 \pmod{2}, \quad r\in \{3,5\}, \label{N24-8nr-cong} \\
N^{0}(2,4;16n+1) &\equiv 0 \pmod{16}, \label{N24-16n1-cong} \\
N^{0}(2,4;16n+9)&\equiv 0 \pmod{8}, \label{N24-16n9-cong} \\
N^{0}(2,4;16n+11)&\equiv 0 \pmod{16}, \label{N24-16n11-cong}\\
N^{0}(2,4;16n+13)&\equiv 0 \pmod{4}. \label{N24-16n13-cong}
%N^{0}(0,8;8n+r)&\equiv 0 \pmod{2}, \quad r\in \{5,7\},\label{N08-8n57-cong} \\
%N^{0}(0,8;16n+13) &\equiv 0 \pmod{4}, \label{N08-16n-13-cong} \\
%N^{0}(0,8;16n+15)&\equiv 0 \pmod{16}, \label{N08-16n15-cong} \\
%N^{0}(2,8;8n+1)&\equiv 0 \pmod{4}, \label{N28-8n1-mod4} \\
%N^{0}(2,8;16n+11)&\equiv 0 \pmod{8}, \label{N28-16n11-cong} \\
%N^{0}(2,8;16n+13)&\equiv 0 \pmod{2}. \label{N28-16n13-cong}
\end{align}
\end{theorem}
\begin{rem}
(1) From \eqref{N-odd-even} we have $N^{0}(1,4;n)=N^{0}(3,4;n)$ and hence
\begin{align}
N^{0}(1,2;n)=N^{0}(1,4;n)+N^{0}(3,4;n)=2N^{0}(1,4;n). \label{add-N12-N14-N34}
\end{align}
So congruences \eqref{N12-8n4-cong}--\eqref{N12-mod5-cong} also imply the following congruences
\begin{align}
N^{0}(1,4;8n+4) &\equiv 0 \pmod{2}, \label{N14-8n4-cong} \\
N^{0}(1,4;8n+6) &\equiv 0 \pmod{4},  \label{N14-8n6-cong} \\
N^{0}(1,4;40n+r)&\equiv 0 \pmod{5}, \quad r\in \{28,36\}. \label{N14-mod5-cong}
\end{align}
Similarly, since $N^{0}(2,8;n)=N^{0}(6,8;n)$ we have
\begin{align}
N^{0}(2,4;n)=2N^{0}(2,8;n)=2N^{0}(6,8;n). \label{add-N24-N28-N68}
\end{align}
Thus congruences \eqref{N24-16n1-cong}--\eqref{N24-16n13-cong} also imply congruences for $N^{0}(a,8;n)$ with $a\in \{2,6\}$.

(2) In the same way, from \eqref{N08-48-8n57} and \eqref{N04-08-8n57} we know that  congruences \eqref{N04-8n57-cong}--\eqref{N04-32n31-cong} also imply congruences for $N^{0}(a,8;8n+r)$ with $a\in \{0,4\}$ and $r\in \{5,7\}$.
 Moreover, from  \eqref{add-N12-N14-N34}, \eqref{N18-38-8n46} and \eqref{N14-18-8n46}, we see that congruences \eqref{N12-8n4-cong}--\eqref{N12-mod5-cong} also imply some congruences for $N^{0}(a,8;8n+r)$ with $a\in \{1,3\}$ and $r\in \{4,6\}$. We omit these congruences.
\end{rem}

The paper is organized as follows. In Section \ref{sec-pre} we collect some formulas which will be used frequently in our proofs. In Section \ref{sec-modulo} we present the generating functions for $N^{0}(a,M;n)$ ($0\le a \le M-1$ and $M\in \{2,4,8\}$) and their $\ell$-dissections with $\ell \in \{2,4,8\}$.  Section \ref{sec-arith} is devoted to the proofs of all the theorems and Andrews' conjectures. We also point out that many congruences in the literature including \eqref{p-omega-mod5}--\eqref{p-omega-3} are direct consequences of our theorems.
\begin{rem}
Andrews proposed 14 open problems and two conjectures in \cite{Andrews}, which fall into three basic groups: combinatorial, asymptotic and congruential. Many of these problems have been solved in the past ten years. For example, Garvan \cite{Garvan} solved Andrews' Problem 12. Ji \cite{Ji} found combinatorial proofs of some identities in Andrews' paper and thus provided answers to most of Andrews' problems in the combinatorial group. Bringmann \cite{Bringmann} gave answers to Andrews' problems 11 and 13. In particular, she found an asymptotic formula for $\eta_{2}(n)$. Moreover, Bringmann, Lovejoy and Osburn \cite{BLO} gave two-parameter generalizations to the $k$-th symmetrized rank moment and the $k$-marked Durfee symbol and studied the automorphic properties of their generating functions.   Here we solve Andrews' Conjectures A and B \cite{Andrews}, which deal with explicit congruences related to the moments of ranks and odd ranks.
\end{rem}
\section{Preliminaries}\label{sec-pre}
From \cite[Eq.\ (4.2)]{HM-PLMS} we find
\begin{align}\label{g-exp}
g(x;q)=\sum_{n=0}^{\infty}\frac{q^{n(n+1)}}{(x;q)_{n+1}(q/x;q)_{n+1}}.
\end{align}
It follows that
\begin{align}\label{g-symmetric}
g(x;q)=g(q/x;q).
\end{align}
For example, we have
\begin{align}\label{iq-minus}
g(iq;q^2)=g(-iq;q^2).
\end{align}
Comparing \eqref{odd-rank-gen} with \eqref{g-exp} we obtain
\begin{align}\label{g-R}
R_{1}^{0}(z;q)=qg(zq;q^2).
\end{align}

We define
\begin{align}
j(x;q):=(x;q)_{\infty}(q/x;q)_{\infty}(q;q)_{\infty}=\sum_{n=-\infty}^{\infty}(-1)^nq^{n(n+1)/2}x^n,
\end{align}
where the last equality follows from Jacobi's triple product identity. Meanwhile, we use the following notations:
\begin{align}
J_{a,m}:=j(q^a;q^m), \quad \overline{J}_{a,m}:=j(-q^a;q^m).
\end{align}
The following product rearrangements will be used frequently:
\begin{align}
\overline{J}_{0,1}=2\overline{J}_{1,4}=2\frac{J_{2}^2}{J_{1}}, \quad \overline{J}_{1,2}=\frac{J_{2}^5}{J_{1}^2J_{4}^2}, \quad J_{1,2}=\frac{J_{1}^2}{J_2}, \quad J_{1,4}=\frac{J_1J_4}{J_2}.
\end{align}
\begin{lemma}\label{g-lemma}
(Cf.\ \cite[p.\ 32]{Ramanujan}, \cite[(12.5.3)]{Lost1}.) For a generic $x\in \mathbb{C}$ we have
\begin{align}\label{g-transform}
g(x;q)=-x^{-1}+qx^{-3}g(-qx^{-2};q^4)-qg(-qx^2;q^4)+\frac{J_{2}^5}{xj(x;q)j(-qx^2;q^2)J_{4}^2}.
\end{align}
\end{lemma}
\begin{lemma}\label{g-sum-minus}
(Cf.\ \cite[p.\ 39]{Ramanujan}, \cite[(12.4.4)]{Lost1}.) For a generic $x\in \mathbb{C}$ we have
\begin{align}
g(x;q)+g(-x;q)=-2qg(-qx^2;q^4)+\frac{2J_{2}^{5}}{j(-qx^2;q^4)j(x^2;q^2)J_{1}^2}, \label{g-sum}\\
g(x;q)-g(-x;q)=-2x^{-1}+2qx^{-3}g(-qx^{-2};q^4)+\frac{2J_{2}^5}{xj(-q^3x^2;q^4)j(x^2;q^2)J_{1}^2}. \label{g-minus}
\end{align}
\end{lemma}
It seems that the completed set of identities \eqref{g-transform}--\eqref{g-minus} first appeared in \cite{Mortenson-Rama}.

As some consequences of Lemma \ref{g-sum-minus}, if we replace $q$ by $q^2$ and set $x=q$ in \eqref{g-sum} (resp.\ \eqref{g-minus}), we obtain
\begin{align}\label{q-sum}
g(q;q^2)+g(-q;q^2)=-2q^2g(-q^4;q^8)+2\frac{J_4^{8}J_{16}^2}{J_{2}^4J_{8}^5}
\end{align}
and
\begin{align}\label{q-minus}
g(q;q^2)-g(-q;q^2)=-2q^{-1}+2q^{-1}g(-1;q^8)+q^{-1}\frac{J_{4}^6J_{8}}{J_{2}^4J_{16}^2},
\end{align}
respectively.

We let $\zeta_M=e^{2\pi i/M}$ throughout this paper. In the same way, we can deduce the following identities, which will be used in Section \ref{sec-modulo}:
\begin{align}
g(iq;q^2)+g(-iq;q^2)&=-2q^2g(q^4;q^8)+2\frac{J_{8}^3}{J_{4}^2}, \label{iq-sum} \\
g(i;q)-g(-i;q)&=2i+2iqg(q;q^4)-i\frac{J_{2}^7}{J_{1}^3J_{4}^3}, \label{i-minus} \\
g(\zeta_8q;q^2)+g(-\zeta_8q;q^2)&=-2q^2g(-iq^4;q^8)+2\frac{J_{4}^5J_{32}}{J_{2}^2J_{8}^2J_{16}}, \label{zetaq-sum} \\
g(\zeta_8q;q^2)-g(-\zeta_8q;q^2)&=-2\zeta_{8}^{-1}q^{-1}-2\zeta_{8}q^{-1}g(i;q^8)+(\zeta_{8}+\zeta_{8}^{-1})q^{-1}\frac{J_{4}^5J_{16}^2}{J_{2}^2J_{8}^3J_{32}}, \label{zetaq-minus} \\
g(i\zeta_8q;q^2)+g(-i\zeta_8q;q^2)&=-2q^2g(iq^4;q^8)+2\frac{J_{4}^5J_{32}}{J_{2}^2J_{8}^2J_{16}}, \label{izetaq-sum} \\
g(i\zeta_8q;q^2)-g(-i\zeta_8q;q^2)&=2\zeta_8q^{-1}+2\zeta_{8}^{-1}q^{-1}g(-i;q^8)-(\zeta_8+\zeta_{8}^{-1})q^{-1}\frac{J_{4}^5J_{16}^2}{J_{2}^2J_{8}^3J_{32}}. \label{izetaq-minus}
\end{align}

\begin{lemma}\label{2-dissection}
We have
\begin{align}
J_{1}^2&=\frac{J_{2}J_{8}^5}{J_{4}^2J_{16}^2}-2q\frac{J_{2}J_{16}^2}{J_{8}}, \label{J2} \\
J_{1}^4&=\frac{J_{4}^{10}}{J_{2}^2J_{8}^4}-4q\frac{J_{2}^2J_{8}^4}{J_{4}^2}, \label{J4}\\
\frac{1}{J_{1}^2}&=\frac{J_{8}^5}{J_{2}^5J_{16}^2}+2q\frac{J_{4}^2J_{16}^2}{J_{2}^5J_{8}}, \quad \text{\rm{and}} \label{J-2} \\
\frac{1}{J_{1}^4}&=\frac{J_{4}^{14}}{J_{2}^{14}J_{8}^4}+4q\frac{J_{4}^2J_{8}^4}{J_{2}^{10}}. \label{J-4}
\end{align}
\end{lemma}
\begin{proof}
Recall two important theta functions (see \cite[p.\ 36, Entry 22]{Notebook-3})
\begin{align*}
\phi(q):=&\sum_{n=-\infty}^{\infty}q^{n^2}=\frac{J_{2}^5}{J_{1}^2J_{4}^2}, \\
\psi(q):=&\sum_{n=0}^{\infty}q^{n(n+1)/2}=\frac{J_{2}^2}{J_1}.
\end{align*}
From \cite[Entry 25 (v), (vi)]{Notebook-3} we find
\begin{align}
\phi(q)=&\phi(q^4)+2q\psi(q^8), \label{phi}\\
\phi^2(q)=&\phi^2(q^2)+4q\psi^2(q^4). \label{phi-2}
\end{align}
These two identities immediately lead to \eqref{J-2} and \eqref{J-4}. Next, replacing $q$ by $-q$ in \eqref{phi} and \eqref{phi-2} we obtain
\begin{align}
\phi(-q)=&\phi(q^4)-2q\psi(q^8), \label{phi-negative}\\
\phi^2(-q)=&\phi^2(q^2)-4q\psi^2(q^4). \label{phi-negative-2}
\end{align}
Since
\begin{align}
\phi(-q)=\frac{J_{1}^2}{J_2},
\end{align}
 \eqref{J2} and \eqref{J4} follow immediately from \eqref{phi-negative} and \eqref{phi-negative-2}.
\end{proof}
\begin{corollary}\label{eta-68-cor}
We have
\begin{align}
\frac{1}{J_{1}^6}=\left(\frac{J_8J_{4}^{14}}{J_{2}^{19}J_{16}^2}+8q^2\frac{J_{4}^4J_{8}^3J_{16}^2}{J_{2}^{15} } \right)+2q\left(\frac{J_{4}^{16}J_{16}^2}{J_{2}^{19}J_{8}^5}+2\frac{J_{4}^2J_{8}^9}{J_{2}^{15}J_{16}^2} \right) \label{J-6}
\end{align}
and
\begin{align}
\frac{1}{J_{1}^8}=\left(\frac{J_{4}^{28}}{J_{2}^{28}J_{8}^8}+16q^2\frac{J_{4}^4J_{8}^8}{J_{2}^{20}}\right)+8q\frac{J_{4}^{16}}{J_{2}^{24}}. \label{J-8}
\end{align}
\end{corollary}
\begin{proof}
Multiplying \eqref{J-2} and \eqref{J-4}, we get \eqref{J-6}. Taking squares on both sides of \eqref{J-4}, we get \eqref{J-8}.
\end{proof}

\section{Odd Ranks Modulo 2, 4 and 8}\label{sec-modulo}
In this section, we obtain the generating functions of $N^{0}(a,M;\ell n+r)$ with $M, \ell \in \{2,4,8\}$, $0\le a \le M-1$ and  $0 \le r \le \ell-1$. From \eqref{N-odd-even} we know that we only need to consider $0\le a \le \frac{M}{2}$ since $N^{0}(a,M;n)=N^{0}(M-a;M;n)$. Moreover, from Definition \ref{odd-durfee-defn} it is clear that $N^0(m,n)=0$ when $n\equiv m$ (mod 2). Thus when $M, \ell \in \{2,4,8\}$, $N^{0}(a,M;\ell n+r)=0$ if $a$ and $r$ have the same parity. Therefore, we only need to consider the case when $a\equiv r+1$ (mod 2).

To get the generating function for $N^{0}(a,M;n)$, using \eqref{odd-rank-gen} and the fact that
\begin{align}
\frac{1}{M}\sum_{j=0}^{M-1}\zeta_{M}^{kj}=\left\{\begin{aligned} &1, & k\equiv 0 \pmod{M}, \\  &0, & k\not\equiv 0 \pmod{M}, \end{aligned} \right.
\end{align}
we obtain the following identity:
\begin{align}\label{Rank-general}
\sum_{n=1}^{\infty}N^{0}(a,M;n)q^n=\frac{1}{M}\sum_{j=0}^{M-1}\zeta_{M}^{-aj}R_{1}^{0}(\zeta_{M}^{j};q).
\end{align}
From \eqref{g-R} we deduce that
\begin{align}\label{start}
\sum_{n=1}^{\infty}N^{0}(a,M;n)q^n=\frac{q}{M}\sum_{j=0}^{M-1}\zeta_{M}^{-aj}g(\zeta_{M}^{j}q;q^2).
\end{align}

Now we consider the odd rank modulo 2.
\begin{theorem}\label{thm-mod2}
We have
\begin{align}
&\sum_{n=0}^{\infty}N^{0}(0,2;2n+1)q^n=-qg(-q^2;q^4)+\frac{J_{2}^8J_{8}^2}{J_{1}^4J_{4}^5}, \label{N02-2n1}\\
&\sum_{n=0}^{\infty}N^{0}(0,2;4n+1)q^n=\frac{J_{2}^9}{J_{1}^6J_{4}^2}, \label{N02-4n1}\\
&\sum_{n=0}^{\infty}N^{0}(0,2;4n+3)q^n=-g(-q;q^2)+4\frac{J_{4}^6}{J_{1}^2J_{2}^3}, \label{N02-4n3} \\
&\sum_{n=0}^{\infty}N^{0}(1,2;2n)q^n=-1+g(-1;q^4)+\frac{1}{2}\frac{J_{2}^6J_{4}}{J_{1}^4J_{8}^2}, \label{N12-2n}\\
&\sum_{n=0}^{\infty}N^{0}(1,2;4n)q^n=-1+g(-1;q^2)+\frac{1}{2}\frac{J_{2}^{15}}{J_{1}^8J_{4}^6}, \label{N12-4n}\\
&\sum_{n=0}^{\infty}N^{0}(1,2;4n+2)q^n=2\frac{J_{2}^3J_{4}^2}{J_{1}^4}, \label{N12-4n2}\\
&\sum_{n=0}^{\infty}N^{0}(0,2;8n+1)q^n=\frac{J_{2}^{12}J_{4}}{J_{1}^{10}J_{8}^2}+8q\frac{J_{2}^2J_{4}^3J_{8}^2}{J_{1}^6}, \label{N02-8n1} \\
&\sum_{n=0}^{\infty}N^{0}(0,2;8n+3)q^n=qg(-q^2;q^4)-\frac{J_{2}^8J_{8}^2}{J_{1}^4J_{4}^5}+4\frac{J_{2}^6J_{4}^5}{J_{1}^8J_{8}^2},\label{N02-8n3} \\
&\sum_{n=0}^{\infty}N^{0}(0,2;8n+5)q^n=2\left(\frac{J_{2}^{14}J_{8}^2}{J_{1}^{10}J_{4}^5}+2\frac{J_{4}^9}{J_{1}^6J_{8}^2} \right), \label{N02-8n5} \\
&\sum_{n=0}^{\infty}N^{0}(0,2;8n+7)q^n=-q^{-1}+q^{-1}g(-1;q^4)+\frac{1}{2}q^{-1}\frac{J_{2}^6J_{4}}{J_{1}^4J_{8}^2}+8\frac{J_{2}^8J_{8}^2}{J_{1}^8J_{4}}, \label{N02-8n7} \\
&\sum_{n=0}^{\infty}N^{0}(1,2;8n)q^n=-1+g(-1;q)+\frac{1}{2}\frac{J_{2}^{22}}{J_{1}^{13}J_{4}^8}+8q\frac{J_{4}^8}{J_{1}^5J_{2}^2}, \label{N12-8n} \\
&\sum_{n=0}^{\infty}N^{0}(1,2;8n+2)q^n=2\frac{J_{2}^{16}}{J_{1}^{11}J_{4}^4}, \label{N12-8n2} \\
&\sum_{n=0}^{\infty}N^{0}(1,2;8n+4)q^n=4\frac{J_{2}^{10}}{J_{1}^9}, \label{N12-8n4} \\
&\sum_{n=0}^{\infty}N^{0}(1,2;8n+6)q^n=8\frac{J_{2}^4J_{4}^4}{J_{1}^7}. \label{N12-8n6}
\end{align}
\end{theorem}
\begin{proof}
Setting $(a,M)=(0,2)$ in \eqref{start} we obtain
\begin{align}
\sum_{n=0}^{\infty}N^{0}(0,2;n)q^n=\frac{q}{2}\left(g(q;q^2)+g(-q;q^2) \right). \label{N02-start}
\end{align}
Substituting \eqref{q-sum} into \eqref{N02-start}, extracting the odd power terms, dividing by $q$ and replacing $q^2$ by $q$, we obtain \eqref{N02-2n1}. Next, substituting \eqref{J-4} into \eqref{N02-2n1} and extracting the even (resp.\ odd) power terms, we obtain \eqref{N02-4n1} (resp.\ \eqref{N02-4n3}). Now substituting \eqref{J-6} into \eqref{N02-4n1} and extracting the even (resp.\ odd) power terms, we obtain \eqref{N02-8n1} (resp.\ \eqref{N02-8n5}). Subtracting \eqref{q-minus} from \eqref{q-sum}, we obtain
\begin{align}\label{g-q-q2}
g(-q;q^2)=q^{-1}-q^{-1}g(-1;q^8)-q^2g(-q^4;q^8)+\frac{J_{4}^8J_{16}^2}{J_{2}^4J_{8}^5}-\frac{1}{2}q^{-1}\frac{J_{4}^6J_{8}}{J_{2}^4J_{16}^2}.
\end{align}
Substituting \eqref{J-2} and \eqref{g-q-q2} into \eqref{N02-4n3} and extracting the even (resp.\ odd) power terms, we obtain \eqref{N02-8n3} (resp.\ \eqref{N02-8n7}).

Similarly, setting $(a,M)=(1,2)$ in \eqref{start} we obtain
\begin{align}
\sum_{n=0}^{\infty}N^{0}(1,2;n)q^n=\frac{q}{2}\left(g(q;q^2)-g(-q;q^2) \right). \label{N12-start}
\end{align}
Substituting \eqref{q-minus} into \eqref{N12-start} and extracting the even power terms, we obtain \eqref{N12-2n}. Next, substituting \eqref{J-4} into \eqref{N12-2n} and extracting the even (resp.\ odd) power terms, we obtain \eqref{N12-4n} (resp.\ \eqref{N12-4n2}). Again, substituting \eqref{J-8} into \eqref{N12-4n} and extracting the even (resp.\ odd) power terms, we obtain \eqref{N12-8n} (resp.\ \eqref{N12-8n4}). Finally, substituting \eqref{J-4} into \eqref{N12-4n2} and extracting the even (resp.\ odd) power terms, we obtain \eqref{N12-8n2} (resp.\ \eqref{N12-8n6}).
\end{proof}

Next we consider the odd rank modulo 4.
\begin{theorem}\label{thm-mod4}
We have
\begin{enumerate}[$(1)$]
\item
\begin{align}
&\sum_{n=0}^{\infty}N^{0}(0,4;2n+1)q^n=q^5g(-q^8;q^{16})-q\frac{J_{8}^8J_{32}^2}{J_{4}^4J_{16}^5}+\frac{1}{2}\left(\frac{J_{2}^8J_{8}^2}{J_{1}^4J_{4}^5}+\frac{J_{4}^3}{J_{2}^2} \right), \label{N04-2n1}\\
&\sum_{n=0}^{\infty}N^{0}(0,4;4n+1)q^n=\frac{J_{2}^4J_{8}^5}{J_{1}^4J_{4}^2J_{16}^2}, \label{N04-4n1} \\
&\sum_{n=0}^{\infty}N^{0}(0,4;4n+3)q^n=q^2g(-q^4;q^8)-\frac{J_{4}^8J_{16}^2}{J_{2}^4J_{8}^5}+2\frac{J_{4}^6}{J_{1}^2J_{2}^3}, \label{N04-4n3} \\
&\sum_{n=0}^{\infty}N^{0}(0,4;8n+1)q^n=\frac{J_{2}^{12}J_{4}}{J_{1}^{10}J_{8}^2}, \label{N04-8n1} \\
&\sum_{n=0}^{\infty}N^{0}(0,4;8n+3)q^n=qg(-q^2;q^4)+2\frac{J_{2}^6J_{4}^5}{J_{1}^8J_{8}^2}-\frac{J_{2}^8J_{8}^2}{J_{1}^4J_{4}^5}, \label{N04-8n3} \\
&\sum_{n=0}^{\infty}N^{0}(0,4;8n+5)q^n=4\frac{J_{4}^9}{J_{1}^6J_{8}^2}, \label{N04-8n5} \\
&\sum_{n=0}^{\infty}N^{0}(0,4;8n+7)q^n=4\frac{J_{2}^8J_{8}^2}{J_{1}^8J_{4}}. \label{N04-8n7}
\end{align}
\item
\begin{align}
&\sum_{n=0}^{\infty}N^{0}(1,4;2n)q^n=-\frac{1}{2}+\frac{1}{2}g(-1;q^4)+\frac{1}{4}\frac{J_{2}^6J_{4}}{J_{1}^4J_{8}^2}, \label{N14-2n} \\
&\sum_{n=0}^{\infty}N^{0}(1,4;4n)q^n=-\frac{1}{2}+\frac{1}{2}g(-1;q^2)+\frac{1}{4}\frac{J_{2}^{15}}{J_{1}^8J_{4}^6}, \label{N14-4n} \\
&\sum_{n=0}^{\infty}N^{0}(1,4;4n+2)q^n=\frac{J_{2}^3J_{4}^2}{J_{1}^4}, \label{N14-4n2} \\
&\sum_{n=0}^{\infty}N^{0}(1,4;8n)q^n=-\frac{1}{2}+\frac{1}{2}g(-1;q)+\frac{1}{4}\frac{J_{2}^{22}}{J_{1}^{13}J_{4}^8}+4q\frac{J_{4}^8}{J_{1}^5J_{2}^2}, \label{N14-8n} \\
&\sum_{n=0}^{\infty}N^{0}(1,4;8n+2)q^n=\frac{J_{2}^{16}}{J_{1}^{11}J_{4}^4}, \label{N14-8n2} \\
&\sum_{n=0}^{\infty}N^{0}(1,4;8n+4)q^n=2\frac{J_{2}^{10}}{J_{1}^9}, \label{N14-8n4} \\
&\sum_{n=0}^{\infty}N^{0}(1,4;8n+6)q^n=4\frac{J_{2}^4J_{4}^4}{J_{1}^7}. \label{N14-8n6}
\end{align}
\item
\begin{align}
&\sum_{n=0}^{\infty}N^{0}(2,4;2n+1)q^n=-q^{-1}+q^{-1}g(-1;q^{16})+\frac{1}{2}q^{-1}\frac{J_{8}^{6}J_{16}}{J_{4}^4J_{32}^2}-\frac{1}{2}\frac{J_{4}^3}{J_{2}^2}+\frac{1}{2}\frac{J_{2}^8J_{8}^2}{J_{1}^4J_{4}^5}, \label{N24-2n1} \\
&\sum_{n=0}^{\infty}N^{0}(2,4;4n+1)q^n=2q\frac{J_{2}^4J_{16}^2}{J_{1}^4J_{8}}, \label{N24-4n1} \\
&\sum_{n=0}^{\infty}N^{0}(2,4;4n+3)q^n=-q^{-1}+q^{-1}g(-1;q^8)+\frac{1}{2}q^{-1}\frac{J_{4}^6J_{8}}{J_{2}^4J_{16}^2}+2\frac{J_{4}^6}{J_{1}^2J_{2}^3}, \label{N24-4n3} \\
&\sum_{n=0}^{\infty}N^{0}(2,4;8n+1)q^n=8q\frac{J_{2}^2J_{4}^3J_{8}^2}{J_{1}^6}, \label{N24-8n1} \\
&\sum_{n=0}^{\infty}N^{0}(2,4;8n+3)q^n=2\frac{J_{2}^6J_{4}^5}{J_{1}^8J_{8}^2}, \label{N24-8n3} \\
&\sum_{n=0}^{\infty}N^{0}(2,4;8n+5)q^n=2\frac{J_{2}^{14}J_{8}^2}{J_{1}^{10}J_{4}^5}, \label{N24-8n5} \\
&\sum_{n=0}^{\infty}N^{0}(2,4;8n+7)q^n=-q^{-1}+q^{-1}g(-1;q^4)+\frac{1}{2}q^{-1}\frac{J_{2}^6J_4}{J_{1}^4J_{8}^2} +4\frac{J_{2}^8J_{8}^2}{J_{1}^8J_{4}}. \label{N24-8n7}
\end{align}
\end{enumerate}
\end{theorem}
\begin{proof}
(1) Setting $(a,M)=(0,4)$ in \eqref{start} we obtain
\begin{align}
\sum_{n=0}^{\infty}N^{0}(0,4;n)q^n=\frac{q}{4}\left(\left(g(q;q^2)+g(-q;q^2) \right)+\left(g(iq;q^2)+g(-iq;q^2) \right) \right). \label{N04-start}
\end{align}
Substituting \eqref{q-sum} and \eqref{iq-sum} into \eqref{N04-start}, we get
\begin{align} \label{N04-middle}
\sum_{n=0}^{\infty}N^{0}(0,4;n)q^n=&\frac{q}{4}\Big(-2q^2\left(g(q^4;q^8)+g(-q^4;q^8) \right) +2\frac{J_{4}^8J_{16}^2}{J_{2}^4J_{8}^5}+2\frac{J_{8}^3}{J_{4}^2}  \Big).
\end{align}
Replacing $q$ by $q^4$ in \eqref{q-sum}, we obtain
\begin{align}
g(q^4;q^8)+g(-q^4;q^8)=-2q^8g(-q^{16};q^{32})+2\frac{J_{16}^8J_{64}^2}{J_{8}^4J_{32}^5}. \label{q4-sum}
\end{align}
From \eqref{q4-sum} we see that \eqref{N04-middle} reduces to
\begin{align}\label{N04-odd-gen}
\sum_{n=0}^{\infty}N^{0}(0,4;2n+1)q^n&=q^5g(-q^8;q^{16})-q\frac{J_{8}^8J_{32}^2}{J_{4}^4J_{16}^5}
+\frac{1}{2}\left(\frac{J_{2}^8J_{8}^2}{J_{1}^4J_{4}^5}+\frac{J_{4}^3}{J_{2}^2} \right).
\end{align}
This proves \eqref{N04-2n1}.

Substituting \eqref{J-4} into \eqref{N04-odd-gen}, we arrive at
\begin{align}\label{add-1}
\sum_{n=0}^{\infty}N^{0}(0,4;2n+1)q^n=q^5g(-q^8;q^{16})-q\frac{J_{8}^8J_{32}^2}{J_{4}^4J_{16}^5}+\frac{1}{2}\left(\frac{J_{4}^3}{J_{2}^2}+\frac{J_{4}^9}{J_{2}^6J_{8}^2} \right)+2q\frac{J_{8}^6}{J_{2}^2J_{4}^3}.
\end{align}
Extracting the even power terms,  we obtain
\begin{align*}
&\sum_{n=0}^{\infty}N^{0}(0,4;4n+1)q^n\nonumber \\
=&~ \frac{1}{2}\left(\frac{J_{2}^3}{J_{1}^2}+\frac{J_{2}^9}{J_{1}^6J_{4}^2} \right) \\
=&~ \frac{1}{2} \frac{J_{2}^3}{J_{1}^4J_{4}^2}\left(J_{1}^2J_{4}^2+\frac{J_{2}^6}{J_{1}^2} \right) \\
=&~ \frac{1}{2} \frac{J_{2}^3}{J_{1}^4J_{4}^2}\left(\left(\frac{J_{2}J_{8}^5}{J_{4}^2J_{16}^2}-2q\frac{J_{2}J_{16}^2}{J_{8}} \right)J_{4}^2+J_{2}^6\left(\frac{J_{8}^5}{J_{2}^5J_{16}^2}+2q\frac{J_{4}^2J_{16}^2}{J_{2}^5J_{8}} \right) \right) \\
=&~ \frac{J_{2}^4J_{8}^5}{J_{1}^4J_{4}^2J_{16}^2},
\end{align*}
where in the last second line we have used \eqref{J2} and \eqref{J-2}. This proves \eqref{N04-4n1}. Now we substitute \eqref{J-4} into \eqref{N04-4n1} and extracting the even (resp.\ odd) power terms, we obtain \eqref{N04-8n1} (resp.\ \eqref{N04-8n5}).

Extracting the odd power terms in \eqref{add-1}, we obtain
\begin{align}\label{N04-4n3-gen}
\sum_{n=0}^{\infty}N^{0}(0,4;4n+3)q^n=q^2g(-q^4;q^8)-\frac{J_{4}^8J_{16}^2}{J_{2}^4J_{8}^5}+2\frac{J_{4}^6}{J_{1}^2J_{2}^3}.
\end{align}
Substituting \eqref{J-2} into \eqref{N04-4n3-gen} and extracting the even (resp.\ odd) power terms, we get \eqref{N04-8n3} (resp.\ \eqref{N04-8n7}).

(2) These identities can be proved in the same way as (1). Alternatively, using \eqref{add-N12-N14-N34}, we see that identities \eqref{N14-2n}--\eqref{N14-8n6} follow directly from \eqref{N12-2n}--\eqref{N12-4n2} and \eqref{N12-8n}--\eqref{N12-8n6}.

(3) Setting $(a,M)=(2,4)$ in \eqref{start}, by using \eqref{q-sum} and \eqref{iq-sum} we obtain
\begin{align}
\sum_{n=0}^{\infty}N^{0}(2,4;n)q^n=&\frac{1}{4}q\left(\left(g(q;q^2)+g(-q;q^2) \right)-\left(g(iq;q^2)+g(-iq;q^2) \right)  \right)  \nonumber\\
=&\frac{1}{2}q^3\left(g(q^4;q^8)-g(-q^4;q^8) \right)+\frac{1}{2}q\frac{J_{4}^8J_{16}^2}{J_{2}^4J_{8}^5}-\frac{1}{2}q\frac{J_{8}^3}{J_{4}^2}.\label{N24-start}
\end{align}
Now replacing $q$ by $q^4$ in \eqref{q-minus} and substituting it into \eqref{N24-start}, we obtain \eqref{N24-2n1}.

Substituting \eqref{J-4} into \eqref{N24-2n1}, and extracting the even power terms, we obtain
\begin{align}
&\sum_{n=0}^{\infty}N^{0}(2,4;4n+1)q^n\nonumber \\
=&\frac{1}{2}\left(\frac{J_{2}^9}{J_{1}^6J_{4}^2}-\frac{J_{2}^3}{J_{1}^2} \right) \nonumber \\
=&\frac{1}{2}\frac{J_{2}^3}{J_{1}^4}\left(\frac{J_{2}^6}{J_{1}^2J_{4}^2}-J_{1}^2 \right) \nonumber \\
=&\frac{1}{2}\frac{J_{2}^3}{J_{1}^4}\left(\frac{J_{2}^6}{J_{4}^2}\left(\frac{J_{8}^5}{J_{2}^5J_{16}^2}+2q\frac{J_{4}^2J_{16}^2}{J_{2}^5J_{8}}  \right)-\left(\frac{J_{2}J_{8}^5}{J_{4}^2J_{16}^2}-2q\frac{J_{2}J_{16}^2}{J_{8}} \right) \right) \nonumber \\
=&2q\frac{J_{2}^4J_{16}^2}{J_{1}^4J_{8}},
\end{align}
where in the last second line we have used \eqref{J2} and \eqref{J-2}. This proves \eqref{N24-4n1}. Other identities in (3) can be proved in a similar fashion.
\end{proof}

Finally, we consider the odd rank modulo 8.
\begin{theorem}\label{thm-mod8}
We have
\begin{enumerate}[$(1)$]
\item
\begin{align}
\sum_{n=0}^{\infty}N^{0}(0,8;2n+1)q^n=&-q^{21}g(-q^{32};q^{64})+q^5\frac{J_{32}^8J_{128}^2}{J_{16}^4J_{64}^5}+\frac{1}{4}\frac{J_{2}^8J_{8}^2}{J_{1}^4J_{4}^5}+\frac{1}{4}\frac{J_{4}^3}{J_{2}^2} \nonumber \\
 & +\frac{1}{2}\frac{J_{2}^5J_{16}}{J_{1}^2J_{4}^2J_{8}}-\frac{1}{2}q\frac{J_{8}^8J_{32}^2}{J_{4}^4J_{16}^5}-\frac{1}{2}q\frac{J_{16}^3}{J_{8}^2}, \label{N08-2n1} \\
\sum_{n=0}^{\infty}N^{0}(0,8;4n+1)q^n=&\frac{1}{4}\left(\frac{J_{2}^9}{J_{1}^6J_{4}^2} +\frac{J_{2}^3}{J_{1}^2}+2\frac{J_{4}^4}{J_{2}^2J_{8}}\right), \label{N08-4n1} \\
\sum_{n=0}^{\infty}N^{0}(0,8;4n+3)q^n=&-q^{10}g(-q^{16};q^{32})+q^2\frac{J_{16}^8J_{64}^2}{J_{8}^4J_{32}^5}+\frac{1}{2}\frac{J_{8}^3}{J_{4}^2}+\frac{J_{4}^6}{J_{1}^2J_{2}^3}\nonumber \\
 & -\frac{1}{2}\frac{J_{4}^8J_{16}^2}{J_{2}^4J_{8}^5}, \label{N08-4n3} \\
\sum_{n=0}^{\infty}N^{0}(0,8;8n+1)q^n=&\frac{1}{2}\frac{J_{2}^4}{J_{1}^2J_4}+\frac{1}{4}\frac{J_{4}^5}{J_{1}^2J_{8}^2}
+\frac{1}{4}\frac{J_{2}^{12}J_{4}}{J_{1}^{10}J_{8}^2}+2q\frac{J_{2}^2J_{4}^3J_{8}^2}{J_{1}^6}, \label{N08-8n1} \\
\sum_{n=0}^{\infty}N^{0}(0,8;8n+3)q^n=&-q^5g(-q^8;q^{16})+q\frac{J_{8}^8J_{32}^2}{J_{4}^4J_{16}^5}+\frac{1}{2}\frac{J_{4}^3}{J_{2}^2}-\frac{1}{2}\frac{J_{2}^8J_{8}^2}{J_{1}^4J_{4}^5}
+\frac{J_{2}^6J_{4}^5}{J_{1}^8J_{8}^2}, \label{N08-8n3} \\
\sum_{n=0}^{\infty}N^{0}(0,8;8n+5)q^n=&2\frac{J_{4}^9}{J_{1}^6J_{8}^2}, \label{N08-8n5} \\
\sum_{n=0}^{\infty}N^{0}(0,8;8n+7)q^n=&2\frac{J_{2}^8J_{8}^2}{J_{1}^8J_{4}}. \label{N08-8n7}
\end{align}
\item
\begin{align}
\sum_{n=0}^{\infty}N^{0}(1,8;2n)q^n=&-\frac{1}{4}+\frac{1}{4}g(-1;q^4)+\frac{1}{2}q^4g(q^4;q^{16})
+\frac{1}{8}\frac{J_{2}^6J_{4}}{J_{1}^4J_{8}^2}\nonumber \\
& -\frac{1}{4}\frac{J_{8}^7}{J_{4}^3J_{16}^3}+\frac{1}{4}\frac{J_{2}^5J_{8}^2}{J_{1}^2J_{4}^3J_{16}}, \label{N18-2n} \\
\sum_{n=0}^{\infty}N^{0}(1,8;4n)q^n=&-\frac{1}{4}+\frac{1}{4}g(-1;q^2)+\frac{1}{2}q^2g(q^2;q^8)+\frac{1}{8}\frac{J_{2}^{15}}{J_{1}^8J_{4}^6}, \label{N18-4n} \\
\sum_{n=0}^{\infty}N^{0}(1,8;4n+2)q^n=&\frac{1}{2}\left(\frac{J_{2}^3J_{4}^2}{J_{1}^4}+\frac{J_4J_8}{J_2} \right), \label{N18-4n2} \\
\sum_{n=0}^{\infty}N^{0}(1,8;8n)q^n=&-\frac{1}{4}+\frac{1}{4}g(-1;q)+\frac{1}{2}qg(q;q^4)+\frac{1}{8}\frac{J_{2}^{22}}{J_{1}^{13}J_{4}^8}+2q\frac{J_{4}^8}{J_{1}^5J_{2}^2}, \label{N18-8n} \\
\sum_{n=0}^{\infty}N^{0}(1,8;8n+2)q^n=&\frac{1}{2}\left(\frac{J_{2}^{16}}{J_{1}^{11}J_{4}^4}+\frac{J_2J_4}{J_1} \right), \label{N18-8n2}  \\ \sum_{n=0}^{\infty}N^{0}(1,8;8n+4)q^n=&\frac{J_{2}^{10}}{J_{1}^9}, \label{N18-8n4} \\
\sum_{n=0}^{\infty}N^{0}(1,8;8n+6)q^n=&2\frac{J_{2}^4J_{4}^4}{J_{1}^7}. \label{N18-8n6}
\end{align}
\item
\begin{align}
\sum_{n=0}^{\infty}N^{0}(2,8;2n+1)q^n=&-\frac{1}{2}q^{-1}+\frac{1}{2}q^{-1}g(-1;q^{16})+\frac{1}{4}q^{-1}\frac{J_{8}^{6}J_{16}}{J_{4}^4J_{32}^2}\nonumber\\
&+\frac{1}{4}\frac{J_{2}^8J_{8}^2}{J_{4}^5J_{1}^4}-\frac{1}{4}\frac{J_{4}^3}{J_{2}^2}, \label{N28-2n1}\\
\sum_{n=0}^{\infty}N^{0}(2,8;4n+1)q^n=&q\frac{J_{2}^4J_{16}^2}{J_{1}^4J_{8}}, \label{N28-4n1} \\
\sum_{n=0}^{\infty}N^{0}(2,8;4n+3)q^n=&-\frac{1}{2}q^{-1}+\frac{1}{2}q^{-1}g(-1;q^8)+\frac{1}{4}q^{-1}\frac{J_{4}^6J_{8}}{J_{2}^4J_{16}^2}+\frac{J_{4}^6}{J_{1}^2J_{2}^3}, \label{N28-4n3} \\
\sum_{n=0}^{\infty}N^{0}(2,8;8n+1)q^n=&4q\frac{J_{2}^2J_{4}^3J_{8}^2}{J_{1}^6}, \label{N28-8n1} \\
\sum_{n=0}^{\infty}N^{0}(2,8;8n+3)q^n=&\frac{J_{2}^6J_{4}^5}{J_{1}^8J_{8}^2}, \label{N28-8n3} \\
\sum_{n=0}^{\infty}N^{0}(2,8;8n+5)q^n=&\frac{J_{2}^{14}J_{8}^2}{J_{1}^{10}J_{4}^5}, \label{N28-8n5} \\
\sum_{n=0}^{\infty}N^{0}(2,8;8n+7)q^n=&-\frac{1}{2}q^{-1}+\frac{1}{2}q^{-1}g(-1;q^4)+\frac{1}{4}q^{-1}\frac{J_{2}^6J_{4}}{J_{1}^4J_{8}^2}+2\frac{J_{2}^8J_{8}^2}{J_{1}^8J_{4}}. \label{N28-8n7}
\end{align}
\item
\begin{align}
\sum_{n=0}^{\infty}N^{0}(3,8;2n)q^n=&-\frac{1}{4}+\frac{1}{4}g(-1;q^4)-\frac{1}{2}q^4g(q^4;q^{16})+\frac{1}{4}\frac{J_{8}^7}{J_{4}^3J_{16}^3}\nonumber \\
&-\frac{1}{4}\frac{J_{2}^5J_{8}^2}{J_{1}^2J_{4}^3J_{16}}+\frac{1}{8}\frac{J_{2}^6J_{4}}{J_{1}^4J_{8}^2}, \label{N38-2n}\\
\sum_{n=0}^{\infty}N^{0}(3,8;4n)q^n=&-\frac{1}{4}+\frac{1}{4}g(-1;q^2)-\frac{1}{2}q^2g(q^2;q^8)+\frac{1}{8}\frac{J_{2}^{15}}{J_{1}^8J_{4}^6}, \label{N38-4n} \\
\sum_{n=0}^{\infty}N^{0}(3,8;4n+2)q^n=&\frac{1}{2}\left(\frac{J_{2}^3J_{4}^2}{J_{1}^4}-\frac{J_4J_8}{J_2} \right), \label{N38-4n2} \\
\sum_{n=0}^{\infty}N^{0}(3,8;8n)q^n=&-\frac{1}{4}+\frac{1}{4}g(-1;q)-\frac{1}{2}qg(q;q^4)+\frac{1}{8}\frac{J_{2}^{22}}{J_{1}^{13}J_{4}^8}+2q\frac{J_{4}^8}{J_{1}^5J_{2}^2}, \label{N38-8n} \\
\sum_{n=0}^{\infty}N^{0}(3,8;8n+2)q^n=&\frac{1}{2}\left(\frac{J_{2}^{16}}{J_{1}^{11}J_{4}^4}-\frac{J_2J_4}{J_1} \right), \label{N38-8n2} \\
\sum_{n=0}^{\infty}N^{0}(3,8;8n+4)q^n=&\frac{J_{2}^{10}}{J_{1}^9}, \label{N38-8n4} \\
\sum_{n=0}^{\infty}N^{0}(3,8;8n+6)q^n=&2\frac{J_{2}^4J_{4}^4}{J_{1}^7}. \label{N38-8n6}
\end{align}
\item
\begin{align}
\sum_{n=0}^{\infty}N^{0}(4,8;2n+1)q^n=&q^{-3}-q^{-3}g(-1;q^{64})-\frac{1}{2}q^{-3}\frac{J_{32}^6J_{64}}{J_{16}^4J_{128}^2} \nonumber \\
& +\frac{1}{4}\left(\frac{J_{2}^8J_{8}^2}{J_{1}^4J_{4}^5}-2\frac{J_{2}^5J_{16}}{J_{1}^2J_{4}^2J_{8}}+\frac{J_{4}^3}{J_{2}^2} \right)+\frac{1}{2}q\left( \frac{J_{16}^3}{J_{8}^2}-\frac{J_{8}^8J_{32}^2}{J_{4}^4J_{16}^5}\right), \label{N48-2n1} \\
\sum_{n=0}^{\infty}N^{0}(4,8;4n+1)q^n=&\frac{1}{4}\left(\frac{J_{2}^9}{J_{1}^6J_{4}^2}-2\frac{J_{4}^4}{J_{2}^2J_{8}}+\frac{J_{2}^3}{J_{1}^2} \right), \label{N48-4n1} \\
\sum_{n=0}^{\infty}N^{0}(4,8;4n+3)q^n=&q^{-2}-q^{-2}g(-1;q^{32})-\frac{1}{2}q^{-2}\frac{J_{16}^6J_{32}}{J_{8}^4J_{64}^2}\nonumber \\
& -\frac{1}{2}\frac{J_{8}^3}{J_{4}^2}-\frac{1}{2}\frac{J_{4}^8J_{16}^2}{J_{2}^4J_{8}^5}+\frac{J_{4}^6}{J_{1}^2J_{2}^3}, \label{N48-4n3} \\
\sum_{n=0}^{\infty}N^{0}(4,8;8n+1)q^n=&\frac{1}{4}\frac{J_{2}^{12}J_{4}}{J_{1}^{10}J_{8}^2}+2q\frac{J_{2}^2J_{4}^3J_{8}^2}{J_{1}^6}-\frac{1}{2}\frac{J_{2}^4}{J_{1}^2J_{4}}
+\frac{1}{4}\frac{J_{4}^5}{J_{1}^2J_{8}^2}, \label{N48-8n1} \\
\sum_{n=0}^{\infty}N^{0}(4,8;8n+3)q^n=&q^{-1}-q^{-1}g(-1;q^{16})-\frac{1}{2}q^{-1}\frac{J_{8}^6J_{16}}{J_{4}^4J_{32}^2}-\frac{1}{2}\frac{J_{4}^3}{J_{2}^2}\nonumber \\
&-\frac{1}{2}\frac{J_{2}^8J_{8}^2}{J_{1}^4J_{4}^5}+\frac{J_{2}^6J_{4}^5}{J_{1}^8J_{8}^2}, \label{N48-8n3} \\
\sum_{n=0}^{\infty}N^{0}(4,8;8n+5)q^n=&2\frac{J_{4}^9}{J_{1}^6J_{8}^2}, \label{N48-8n5}\\
\sum_{n=0}^{\infty}N^{0}(4,8;8n+7)q^n=&2\frac{J_{2}^8J_{8}^2}{J_{1}^8J_{4}}.\label{N48-8n7}
\end{align}
\end{enumerate}
\end{theorem}
\begin{proof}
We only give proofs to part (1). Parts (2)-(5) can be proved similarly.

Setting $(a,M)=(0,8)$ in \eqref{start}, we obtain
\begin{align}
\sum_{n=0}^{\infty}N^{0}(0,8;n)q^n=&\frac{1}{8}q\Big(\left(g(q;q^2)+g(-q;q^2) \right)+\left(g(\zeta_8q;q^2)+g(-\zeta_8q;q^2) \right)\nonumber \\
& +\left(g(iq;q^2)+g(-iq;q^2) \right)+\left(g(i\zeta_8q;q^2)+g(-i\zeta_8q;q^2) \right) \Big).
\end{align}
Substituting \eqref{q-sum}, \eqref{iq-sum}, \eqref{zetaq-sum}, and \eqref{izetaq-sum} into the above identity, we deduce that
\begin{align}\label{N08-2nd}
\sum_{n=0}^{\infty}N^{0}(0,8;n)q^n=&-\frac{q^3}{4}\left(\left(g(q^4;q^8)+g(-q^4;q^8) \right)+\left(g(iq^4;q^8)+g(-iq^4;q^8) \right) \right)\nonumber \\
&+\frac{1}{4}q\left(\frac{J_{4}^8J_{16}^2}{J_{2}^4J_{8}^5}+\frac{J_{8}^3}{J_{4}^2}+2\frac{J_{4}^5J_{32}}{J_{2}^2J_{8}^2J_{16}} \right).
\end{align}
Replacing $q$ by $q^4$ in \eqref{q-sum} and \eqref{iq-sum} and substituting them into \eqref{N08-2nd}, then using \eqref{q-sum} with $q$ replaced by $q^{16}$, we arrive at
\begin{align}
\sum_{n=0}^{\infty}N^{0}(0,8;n)q^n=&-\frac{1}{4}q^3\left(4q^{40}g(-q^{64};q^{128})-4q^8\frac{J_{64}^8J_{256}^2}{J_{32}^4J_{128}^5} \right)-\frac{1}{2}q^3\left(\frac{J_{16}^8J_{64}^2}{J_{8}^4J_{32}^5}+\frac{J_{32}^3}{J_{16}^2}\right) \nonumber \\
&+\frac{1}{4}q\left(\frac{J_{4}^8J_{16}^2}{J_{2}^4J_{8}^5}+\frac{J_{8}^3}{J_{4}^2}+2\frac{J_{4}^5J_{32}}{J_{2}^2J_{8}^2J_{16}}  \right).
\end{align}
Since $N^{0}(0,8;2n)=0$, this identity reduces to \eqref{N08-2n1}.

Substituting \eqref{J-2} and \eqref{J-4} into \eqref{N08-2n1}, then extracting the even (resp.\ odd) power terms, we prove \eqref{N08-4n1} (resp.\ \eqref{N08-4n3}).

Next, we substitute  \eqref{J-2} and \eqref{J-6}  into \eqref{N08-4n1}. If we extract the even  power terms, we  get \eqref{N08-8n1} immediately. If we extract the odd power terms, then dividing by $q$ and replacing $q^2$ by $q$, we obtain
\begin{align}\label{revise-proof-1}
\sum_{n=0}^\infty N^{0}(0,8;8n+5)q^n=\frac{1}{2}\left(\frac{J_{2}^2J_{8}^2}{J_{1}^2J_{4}}+\frac{J_{2}^{14}J_{8}^2}{J_{1}^{10}J_{4}^5}+2\frac{J_{4}^9}{J_{1}^6J_{8}^2} \right).
\end{align}
Now we prove that
\begin{align}\label{revise-proof-2}
\frac{J_{2}^2J_{8}^2}{J_{1}^2J_{4}}+\frac{J_{2}^{14}J_{8}^2}{J_{1}^{10}J_{4}^5}=2\frac{J_{4}^9}{J_{1}^6J_{8}^2}.
\end{align}
In fact, using \eqref{J4} and \eqref{J-4} we deduce that
\begin{align}
&\frac{J_{2}^2J_{8}^2}{J_4}J_{1}^4+\frac{J_{2}^{14}J_{8}^2}{J_{1}^4J_{4}^5} \nonumber \\
=&~ \frac{J_{2}^2J_{8}^2}{J_{4}}\left(\frac{J_{4}^{10}}{J_{2}^2J_{8}^4}-4q\frac{J_{2}^2J_{8}^4}{J_{4}^2} \right) +\frac{J_{2}^{14}J_{8}^2}{J_{4}^5}\left(\frac{J_{4}^{14}}{J_{2}^{14}J_{8}^4}+4q\frac{J_{4}^2J_{8}^4}{J_{2}^{10}} \right) \nonumber \\
=&~ 2\frac{J_{4}^9}{J_{8}^2}. \label{revise-proof-3}
\end{align}
Dividing both sides of \eqref{revise-proof-3} by $J_{1}^6$, we obtain \eqref{revise-proof-2}. Substituting \eqref{revise-proof-2} into \eqref{revise-proof-1}, we obtain  \eqref{N08-8n5}.

Finally, substituting \eqref{J-2} into \eqref{N08-4n3}, then extracting the even (resp.\ odd) power terms, we obtain \eqref{N08-8n3} (resp.\ \eqref{N08-8n7}).
\end{proof}
\begin{rem}
All the identities in part (3) of this theorem follow directly from Theorem \ref{thm-mod4} (3) upon using the relation in \eqref{add-N24-N28-N68}. This is similar to the proof of Theorem \ref{thm-mod4} (2).
\end{rem}

We may further give generating functions for $N^{0}(a,M;16n+r)$ for some $0\le a \le M-1$, $0\le r \le 15$ and $M\in \{2,4,8\}$. We give some examples in the following corollary. The proof is straightforward and we omit it.
\begin{corollary}
We have
\begin{align}
&\sum_{n=0}^\infty N^{0}(0,4;16n+5)q^n=4\frac{J_{2}^{23}}{J_{1}^{19}J_{4}J_{8}^2}+32q\frac{J_{2}^{13}J_{4}J_{8}^2}{J_{1}^{15}}, \label{N04-16n5} \\
&\sum_{n=0}^\infty N^{0}(0,4;16n+7)q^n=4\frac{J_{2}^{27}}{J_{1}^{20}J_{4}^6}+64q\frac{J_{2}^3J_{4}^{10}}{J_{1}^{12}}, \label{N04-16n7} \\
&\sum_{n=0}^\infty N^{0}(0,4;16n+13)q^n=8\frac{J_{2}^{25}J_{8}^2}{J_{1}^{19}J_{4}^7}+16\frac{J_{2}^{11}J_{4}^7}{J_{1}^{15}J_{8}^2}, \label{N04-16n13} \\
&\sum_{n=0}^\infty N^{0}(0,4;16n+15)q^n=32\frac{J_{2}^{15}J_{4}^2}{J_{1}^{16}}, \label{N04-16n15} \\
&\sum_{n=0}^\infty N^{0}(2,4;16n+1)q^n=16q\left(\frac{J_{2}^{19}J_{8}^2}{J_{1}^{17}J_{4}^3}+2\frac{J_{2}^5J_{4}^{11}}{J_{1}^{13}J_{8}^2}  \right), \label{N24-16n1} \\
&\sum_{n=0}^\infty N^{0}(2,4;16n+3)q^n=2\frac{J_{2}^{33}}{J_{1}^{22}J_{4}^{10}}+32q\frac{J_{2}^9J_{4}^6}{J_{1}^{14}}, \label{N24-16n3} \\
&\sum_{n=0}^\infty N^{0}(2,4;16n+9)q^n=8\frac{J_{4}^3J_{2}^{17}}{J_{1}^{17}J_{8}^2}+64q\frac{J_{2}^7J_{4}^5J_{8}^2}{J_{1}^{13}}, \label{N24-16n9} \\
&\sum_{n=0}^\infty N^{0}(2,4;16n+11)q^n=16\frac{J_{2}^{21}}{J_{1}^{18}J_{4}^2}. \label{N24-16n11}
%&\sum_{n=0}^{\infty}N^{0}(0,8;16n+13)q^n=4\left(\frac{J_{2}^{25}J_{8}^2}{J_{1}^{19}J_{4}^7}+2\frac{J_{2}^{11}J_{4}^7}{J_{1}^{15}J_{8}^2} \right), \label{N08-16n13} \\
%&\sum_{n=0}^{\infty}N^{0}(0,8;16n+15)q^n=16\frac{J_{2}^{15}J_{4}^2}{J_{1}^{16}}. \label{N08-16n15}
\end{align}
\end{corollary}
%\begin{proof}
%Substituting \eqref{J-6} into \eqref{N04-8n5} and extracting the even  and odd power terms, we obtain \eqref{N04-16n5} and \eqref{N04-16n13}, respectively.
%
%Similarly, substituting \eqref{J-8} into \eqref{N04-8n7},  we obtain \eqref{N04-16n7} and \eqref{N04-16n15}. Substituting \eqref{J-6} into \eqref{N24-8n1}, we obtain \eqref{N24-16n1} and \eqref{N24-16n9}. Substituting \eqref{J-8} into \eqref{N24-8n3},  we obtain \eqref{N24-16n3} and \eqref{N24-16n11}.
%\end{proof}
\begin{rem}
Using \eqref{N08-48-8n57} and \eqref{N04-08-8n57}, identities \eqref{N04-16n5}--\eqref{N04-16n15} also imply similar formulas for $N^{0}(a,8;16n+r)$ with $a\in \{0,4\}$ and $r\in \{5,7,13,15\}$. Similarly, using \eqref{add-N24-N28-N68}, we can also get identities for $N^{0}(a,8;16n+r)$ with $a\in \{2,6\}$ and $r\in \{1,3,9,11\}$ from  identities \eqref{N24-16n1}--\eqref{N24-16n11}.
\end{rem}
%&\sum_{n=0}^{\infty}N^{0}(4,8;16n+3)q^n=\frac{J_{2}^{33}}{J_{1}^{22}J_{4}^{10}}-\frac{1}{2}\frac{J_{2}^3}{J_{1}^2}-\frac{1}{2}\frac{J_{2}^9}{J_{1}^6J_{4}^2}+16q\frac{J_{2}^9J_{4}^6}{J_{1}^{14}}, \label{N48-16n3} \\
%&\sum_{n=0}^{\infty}N^{0}(4,8;16n+11)q^n=q^{-1}-q^{-1}g(-1;q^8)-\frac{1}{2}q^{-1}\frac{J_{4}^6J_{8}}{J_{2}^4J_{16}^2}-2\frac{J_{4}^6}{J_{1}^2J_{2}^3}+8\frac{J_{2}^{21}}{J_{1}^{18}J_{4}^2}, \label{N48-16n11} \\
%&\sum_{n=0}^{\infty}N^{0}(4,8;16n+13)q^n=4\left(\frac{J_{2}^{25}J_{8}^2}{J_{1}^{19}J_{4}^7}+2\frac{J_{2}^{11}J_{4}^7}{J_{1}^{15}J_{8}^2} \right), \label{N48-16n13}\\

\section{Proofs of the Theorems and Andrews' Conjectures}\label{sec-arith}

\begin{proof}[Proof of Theorem \ref{N-relation}]
Comparing \eqref{N08-8n5} with \eqref{N48-8n5}, we obtain the case $r=5$ of \eqref{N08-48-8n57}. Similarly, comparing \eqref{N08-8n7} with \eqref{N48-8n7}, we obtain the case $r=7$ of \eqref{N08-48-8n57}. The equalities in \eqref{N18-38-8n46} can be proved by comparing \eqref{N18-8n4} with \eqref{N38-8n4} and \eqref{N18-8n6} with \eqref{N38-8n6}.
\end{proof}

\begin{proof}[Proof of Corollary \ref{cor-N-relation}]
By \eqref{N08-48-8n57} we deduce that
\begin{align}
N^{0}(0,4;8n+5)=N^{0}(0,8;8n+5)+N^{0}(4,8;8n+5)=2N^{0}(0,8;8n+5).
\end{align}
This proves the case $r=5$ of \eqref{N04-08-8n57}. The rest of the equalities can be proved in a similar fashion.
\end{proof}
\begin{rem}
We can also prove Corollary \ref{cor-N-relation} without using Theorem \ref{N-relation}. Indeed, comparing \eqref{N04-8n5} with \eqref{N08-8n5}, we obtain the case $r=5$ of \eqref{N04-08-8n57} immediately. In the same way we can prove other equalities.
\end{rem}

\begin{proof}[Proof of Theorem \ref{rank-differences}]
(1) Subtracting \eqref{N48-8n1} from \eqref{N08-8n1}, we get \eqref{N08-48-8n1-diff}. Similarly, subtracting \eqref{N38-8n} from \eqref{N18-8n} we get \eqref{N18-38-8n-diff}, and subtracting \eqref{N38-8n2} from \eqref{N18-8n2} we get \eqref{N18-38-8n2-diff}.

The proof of \eqref{N08-48-8n3-diff} requires more tricks. Subtracting \eqref{N48-8n3} from \eqref{N08-8n3} we obtain
\begin{align}
&\sum_{n=0}^{\infty}\left(N^{0}(0,8;8n+3)-N^{0}(4,8;8n+3) \right)q^n \nonumber \\
=&-q^{-1}+q^{-1}g(-1;q^{16})-q^5g(-q^8;q^{16})+\frac{1}{2}q^{-1}\frac{J_{8}^6J_{16}}{J_{4}^4J_{32}^2}+\frac{J_{4}^3}{J_{2}^2}+q\frac{J_{8}^8J_{32}^2}{J_{4}^4J_{16}^5}.  \label{g-add-0} \end{align}
To simplify the right side of the above identity, we invoke Lemma \ref{g-lemma}. Replacing $q$ by $q^2$ and setting $x=q$ in \eqref{g-transform}, after rearrangement, we obtain
\begin{align}
g(-1;q^8)-q^3g(-q^4;q^8)=1+qg(q;q^2)-\frac{1}{2}\frac{J_{2}J_{4}^6}{J_{1}^2J_{8}^4}. \label{g-add-1}
\end{align}
Replacing $q$ by $q^2$ in \eqref{g-add-1} and then substituting it into \eqref{g-add-0}, after rearrangement, we get
\begin{align}
&\sum_{n=0}^{\infty}\left(N^{0}(0,8;8n+3)-N^{0}(4,8;8n+3) \right)q^n \nonumber \\
=&\frac{1}{2}q^{-1}\left(\frac{J_{8}^6J_{16}}{J_{4}^4J_{32}^2}-\frac{J_{4}J_{8}^6}{J_{2}^2J_{16}^4} \right)+\frac{J_{4}^3}{J_{2}^2}+q\frac{J_{8}^8J_{32}^2}{J_{4}^4J_{16}^5}+qg(q^2;q^4), \nonumber \\
=&\frac{1}{2}q^{-1}\left(\frac{J_{8}^6J_{16}}{J_{4}^4J_{32}^2}-\frac{J_{4}J_{8}^6}{J_{16}^4}\left(\frac{J_{16}^5}{J_{4}^5J_{32}^2}+2q^2\frac{J_{8}^2J_{32}^2}{J_{4}^5J_{16}} \right) \right)+\frac{J_{4}^3}{J_{2}^2}+q\frac{J_{8}^8J_{32}^2}{J_{4}^4J_{16}^5}+qg(q^2;q^4) \nonumber \\
=&\frac{J_{4}^3}{J_{2}^2}+qg(q^2;q^4).
 \label{g-add-2}
\end{align}
where in the last second equality we have used \eqref{J-2} with $q$ replaced by $q^2$.
\end{proof}
\begin{proof}[Proof of Corollary \ref{N-p-omega}]
Extracting the odd power terms in \eqref{N08-48-8n3-diff}, we obtain
\begin{align}\label{N08-48-16n15-diff}
\sum_{n=0}^{\infty}\left(N^{0}(0,8;16n+11)-N^{0}(4,8;16n+11) \right)q^n=g(q;q^2).
\end{align}
From \eqref{omega-defn}, \eqref{Fine-formula} and \eqref{g-exp} we know that
\begin{align}\label{p-omega-g}
\sum_{n=0}^{\infty}p_{\omega}(n)=qg(q;q^2).
\end{align}
Comparing \eqref{N08-48-16n15-diff} with \eqref{p-omega-g}, we obtain \eqref{N-pomega-16n5}.
\end{proof}

\begin{proof}[Proof of Theorem \ref{rank-inequality}]
For two series
\begin{align*}
A_1(q):=\sum_{n=-\infty}^{\infty}a_1(n)q^n \quad \textrm{and} \quad A_2(q):=\sum_{n=-\infty}^{\infty}a_{2}(n)q^n,
\end{align*}
we say that $A_{1}(q)\succeq A_{2}(q)$ if $a_{1}(n)\ge a_{2}(n)$ holds for any integer $n$. For example, we have
\begin{align}
\psi^2(q)=\left(1+q+q^3+q^6+\cdots \right)^2\succeq 1+q+q^2+q^3.
\end{align}

We have
\begin{align}
\frac{J_{2}^4}{J_{1}^2J_{4}}=\frac{\psi^2(q)}{J_{4}}\succeq \frac{1+q+q^2+q^3}{1-q^4}=\frac{1}{1-q}=\sum_{n=0}^{\infty}q^{n}.
\end{align}
Hence from \eqref{N08-48-8n1-diff} we deduce that for any $n\ge 0$,
\begin{align}\label{N08-48-8n1-ineq}
N^{0}(0,8;8n+1)>N^{0}(4,8;8n+1).
\end{align}

Similarly, we have
\begin{align}
\frac{J_{2}^3}{J_{1}^2}=\frac{\psi(q)}{(q;q^2)_{\infty}}\succeq \frac{1}{1-q}=\sum_{n=0}^{\infty}q^n.
\end{align}
From \eqref{g-exp} we have
\begin{align}
g(q;q^2)=\sum_{n=0}^{\infty}\frac{q^{2n(n+1)}}{(q;q^2)_{n+1}^{2}}\succeq \frac{1}{1-q}=\sum_{n=0}^{\infty}q^n.
\end{align}
Thus by \eqref{N08-48-8n3-diff} we deduce that for any $n\ge 0$,
\begin{align}\label{N08-48-8n3-ineq}
N^{0}(0,8;8n+3)>N^{0}(4,8;8n+3).
\end{align}
Combining \eqref{N08-48-8n1-ineq} with \eqref{N08-48-8n3-ineq}, we obtain \eqref{N08-48-8n13-ineq}.

By \eqref{g-exp} we have
\begin{align}
qg(q;q^4)=q\sum_{n=0}^{\infty}\frac{q^{4n(n+1)}}{(q;q^4)_{n+1}(q^3;q^4)_{n+1}}\succeq \frac{q}{1-q}=\sum_{n=1}^{\infty}q^n,
\end{align}
which together with \eqref{N18-38-8n-diff} implies that for any $n\ge 1$,
\begin{align}\label{N18-38-8n-ineq}
N^{0}(1,8;8n)>N^{0}(3,8;8n).
\end{align}

Next, we observe that
\begin{align}
\frac{J_{2}J_{4}}{J_{1}}=\frac{\psi(q)}{(q^2;q^4)}\succeq \frac{1+q}{1-q^2}=\sum_{n=0}^{\infty}q^n.
\end{align}
This together with \eqref{N18-38-8n2-diff} implies that for any $n\ge 0$,
\begin{align}\label{N18-38-8n2-ineq}
N^{0}(1,8;8n+2)>N^{0}(3,8;8n+2).
\end{align}
Combining \eqref{N18-38-8n-ineq} with \eqref{N18-38-8n2-ineq}, we obtain \eqref{N18-38-8n02-ineq}.
\end{proof}

\begin{proof}[Proof of Theorem \ref{cong-thm}]
Since $N^{0}(-m,n)=N^{0}(m,n)$, we deduce that
\begin{align*}
N^{0}(k,2k;n)=\sum_{m=-\infty}^{\infty}N^{0}(2km+k,n)=2\sum_{m=0}^{\infty}N^{0}(2km+k,n).
\end{align*}
This proves \eqref{Nk-2k-cong}.

From \eqref{N02-8n1} and the binomial theorem we have
\begin{align}
\sum_{n=0}^\infty N^{0}(0,2;8n+1)q^n\equiv \frac{J_{2}^{12}J_{4}}{J_{2}^5J_{8}^2} \pmod{2},
\end{align}
which yields \eqref{N02-16n9-cong}.

Congruence \eqref{N12-8n4-cong} (resp.\ \eqref{N12-8n6-cong}) follows from \eqref{N12-8n4} (resp.\ \eqref{N12-8n6}).

From \eqref{N12-8n4} and the binomial theorem, we deduce that
\begin{align}\label{N12-mod5-cong-proof}
\sum_{n=0}^{\infty}N^{0}(1,2;8n+4)q^n\equiv 4\frac{J_{10}^2J_{1}}{J_{5}^2}\pmod{5}.
\end{align}
By Euler's pentagonal number theorem, we have
\begin{align}
J_{1}=\sum_{n=-\infty}^{\infty}(-1)^nq^{n(3n+1)/2}.
\end{align}
For any integer $n$, the residue of $\frac{n(3n+1)}{2}$ modulo 5 can only be 0, 1 or 2. Using this fact, \eqref{N12-mod5-cong}  follows from \eqref{N12-mod5-cong-proof}.

Congruence \eqref{N04-8n57-cong} follows from \eqref{N04-8n5} and \eqref{N04-8n7}. Congruences \eqref{N04-16n13-cong} and \eqref{N04-16n15-cong} follow from \eqref{N04-16n13} and \eqref{N04-16n15}, respectively.

From \eqref{N04-16n7} and the binomial theorem, we have
\begin{align}
\sum_{n=0}^\infty N^{0}(0,4;16n+7)q^n\equiv 4\frac{J_{2}^{27}}{J_{2}^{10}J_{4}^6} \pmod{16},
\end{align}
which yields \eqref{N04-32n23-cong}. Similarly, using the binomial theorem we have $J_{1}^{16}\equiv J_{2}^8$ (mod 16). Hence  \eqref{N04-16n15} implies
\begin{align}
\sum_{n=0}^\infty N^{0}(0,4;16n+15)q^n\equiv 32J_{2}^7J_{4}^2 \pmod{512},
\end{align}
which yields \eqref{N04-32n31-cong}.

Congruences \eqref{N24-16n1-cong}, \eqref{N24-16n9-cong} and \eqref{N24-16n11-cong} follow from \eqref{N24-16n1}, \eqref{N24-16n9} and \eqref{N24-16n11}, respectively.

Finally, congruence \eqref{N24-16n13-cong} follows from \eqref{N24-8n5} and the binomial theorem.
\end{proof}
\begin{rem}
As mentioned in the introduction, we can prove some congruences in the literature using Theorems \ref{N-relation} and \ref{cong-thm}. For example,
since $N^{0}(0,2;2n)=0$ and $N^{0}(1,4;2n)=N^{0}(3,4;2n)$, we have
\begin{align}\label{fact}
p_{\omega}(2n)=N^{0}(1,2;2n)=2N^{0}(1,4;2n).
\end{align}
From the case $r=4$ of \eqref{N14-18-8n46} we deduce that
\begin{align}\label{revise-add-1}
p_{\omega}(8n+4)=2N^{0}(1,4;8n+4)=4N^{0}(1,8;8n+4).
\end{align}
This gives a combinatorial interpretation to \eqref{p-omega-1}. Similarly, from the case $r=6$ of \eqref{N14-18-8n46} we have
\begin{align}
p_{\omega}(8n+6)=4N^{0}(1,8;8n+6).
\end{align}
By \eqref{N18-8n6} we immediately get \eqref{p-omega-2}. In the same way, using \eqref{N04-16n13-cong}, \eqref{N24-16n13-cong} and the fact that
\begin{align*}
p_{\omega}(16n+13)=N^{0}(0,4;16n+13)+N^{0}(2,4;16n+13),
\end{align*}
we can prove \eqref{p-omega-3}.

A stronger form of Waldherr's congruence \eqref{p-omega-mod5}, namely,
\begin{align}
p_{\omega}(40n+r)\equiv 0 \pmod{20}, \quad r\in \{28, 36\},
\end{align}
can also be deduced from \eqref{N14-mod5-cong}  in view of \eqref{revise-add-1}.
\end{rem}

Finally, we present proofs for Andrews' conjectures in \cite{Andrews}.
\begin{proof}[Proof of Conjectures \ref{conj-1} and \ref{conj-2}]
From \eqref{D-eta} and \eqref{eta-defn} we deduce that
\begin{align*}
\mathcal{D}_{2}^{0}(n)=\eta_{2}^{0}(n)=\sum_{m=-\infty}^{\infty}\binom{m+1}{2}N^{0}(m,n)\equiv \sum_{\begin{smallmatrix}m=-\infty \\ m\equiv 1,2 \pmod{4}\end{smallmatrix}}^{\infty}N^{0}(m,n) \pmod{2}.
\end{align*}
This implies
\begin{align*}%\label{D2-mod2}
\mathcal{D}_{2}^{0}(n)\equiv N^{0}(1,4;n)+N^{0}(2,4;n) \pmod{2}.
\end{align*}
Therefore, we have $\mathcal{D}_{2}^{0}(2n)\equiv N^{0}(1,4;2n)$ (mod 2). From \eqref{N14-8n4-cong} and \eqref{N14-8n6-cong} we complete the proof of  Conjecture \ref{conj-1}.

Similarly, we have
\begin{align}
\mathcal{D}_{3}^{0}(n)=\eta_{4}^{0}(n)=&\sum_{m=-\infty}^{\infty}\binom{m+2}{4}N^{0}(m,n)\nonumber \\
\equiv &\sum_{\begin{smallmatrix}m=-\infty \\ m\equiv 2,3,4,5 \pmod{8}\end{smallmatrix}}^{\infty}N^{0}(m,n) \pmod{2}.
\end{align}
Hence
\begin{align*}
\mathcal{D}_{3}^{0}(n)\equiv N^{0}(2,8;n)+N^{0}(3,8;n)+N^{0}(4,8;n)+N^{0}(5,8;n) \equiv N^{0}(2,8;n) \pmod{2},
\end{align*}
where we used the facts that $N^{0}(3,8;n)=N^{0}(5,8;n)$ and $N^{0}(4,8;n)\equiv 0$ (mod 2) by \eqref{Nk-2k-cong}.  Now by \eqref{add-N24-N28-N68} we have $N^{0}(2,8;n)=\frac{1}{2}N^{0}(2,4;n)$. By congruences \eqref{N24-16n1-cong}, \eqref{N24-16n9-cong}, \eqref{N24-16n11-cong} and \eqref{N24-16n13-cong},  we complete the proof of Conjecture \ref{conj-2}.
\end{proof}
%\section{Concluding Remarks}
%There are many interesting questions that can be investigated in the future. For example, Bringmann \cite{Bringmann} established an asymptotic formulae for $\eta_{2}(n)$ and showed that for any positive integer $j$ and prime $\ell>3$, there exists infinitely many arithmetic progressions $An+B$ such that
%\begin{align*}
%\eta_{2}(An+B)\equiv 0 \pmod{\ell^{j}}.
%\end{align*}
%Bringmann et al. \cite{BLO} also developed some generalizations to the moments of ranks. In a forthcoming paper, the author will establish numerous explicit congruences of this form for the moments of ranks, cranks, odd ranks and their generalized versions.
% %For example, numerical evidences suggest that the following congruences hold.
%\begin{conj}
%For any integer $n\ge 0$, we have
%\begin{align}
%\eta_{4}(125n+99) &\equiv 0\pmod{25}, \\
%\eta_{4}(625n+224) &\equiv 0 \pmod{125}.
%\end{align}
%\end{conj}
%If we want to use the strategy in the proof of Theorem \ref{eta-conj} to prove this conjecture, then one has to study the rank modulo 25 and 125, which seems to be very complicated.

%We end this paper with the following question: can we prove the relation \eqref{N-pomega-16n5} by establishing bijections between the combinatorial quantities enumerated on both sides?

\subsection*{Acknowledgements}
The author was supported by the National Natural Science Foundation of China (11801424), ``the Fundamental Research Funds for the Central Universities'' (Project No.\ 2042018kf0027) and a start-up research grant  of the Wuhan University. Part of this work was done during the author's stay at the Nanyang Technological University. The author would like to thank Prof. Song Heng Chan and NTU for their support.

\end{document}